\documentclass[10pt,a4paper]{amsart} 
\usepackage{amsaddr}
\usepackage{amssymb,amsthm,amsmath,amsfonts,amscd}
\usepackage{cite}
\usepackage{graphicx}
\usepackage{hyperref} 
\usepackage{url}
\allowdisplaybreaks[1]
\usepackage[english]{babel}
\usepackage{enumerate}
\usepackage{bbm}
\usepackage{stmaryrd}
\usepackage{mathrsfs}

\newcommand{\norm}[1]{\left\lVert#1\right\rVert}

\newcommand{\abs}[1]{\lvert#1\rvert}
\newcommand{\lrabs}[1]{\left\lvert#1\right\rvert}

\newcommand{\Xscr} {{\mathscr X}}

\newcommand{\eps}{{\varepsilon}}
\renewcommand{\phi}{\varphi}

\newcommand{\R}{\mathbbm{R}}
\newcommand{\N}{\mathbbm{N}}

\newcommand{\Lip}{\operatorname{Lip}}

\renewcommand{\le}{\leq}
\renewcommand{\ge}{\geq}

\newcommand{\BIGOP}[1]{\mathop{\mathchoice%
{\raise-0.22em\hbox{\huge $#1$}}%
{\raise-0.05em\hbox{\Large $#1$}}{\hbox{\large $#1$}}{#1}}}

\newcommand{\BIGboxplus}{\mathop{\mathchoice%
{\raise-0.35em\hbox{\huge $\boxplus$}}%
{\raise-0.15em\hbox{\Large $\boxplus$}}{\hbox{\large $\boxplus$}}{\boxplus}}}

\def\leq {\leqslant}
\def\geq {\geqslant}

\numberwithin{equation}{section}

\newtheorem{theorem}{Theorem}[section] 

\theoremstyle{plain}

\newtheorem{prop}[theorem]{Proposition}

\newtheorem*{bem*}{Bemerkung}

\newcommand{\be}{\begin{eqnarray}}
\newcommand{\ee}{\end{eqnarray}}
\newcommand{\ce}{\begin{eqnarray*}}
\newcommand{\de}{\end{eqnarray*}}

\numberwithin{equation}{section}

\begin{document}
\title[Corrigendum to `Convergence of invariant measures ...']{Corrigendum to `Convergence of invariant measures for singular stochastic diffusion equations'}

\author[I. Ciotir]{Ioana Ciotir}
\address{Department of Mathematics, Faculty of Economics and Business Administration, ``Al. I. Cuza'' University, Bd. Carol no. 9--11, Ia\c{s}{}i, Romania}
\email{ioana.ciotir@feaa.uaic.ro}
\author[J. M. T\"olle]{Jonas M. T\"olle}
\address{Institut f\"ur Mathematik, Technische Universit\"at Berlin (MA 7-5)\\ Stra\ss{}e des 17. Juni 136, 10623 Berlin, Germany}
\email{jonasmtoelle@gmail.com}

\begin{abstract}
We correct a few errors that appeared in [Convergence of invariant measures for singular stochastic diffusion equations, Stochastic Process. Appl. {\bf 122} (2012), no. 4, 1998--2017] by I. Ciotir and J.M. T\"olle.
\end{abstract}
 \keywords{Stochastic evolution equation, stochastic diffusion equation, $p$-Laplace equation, $1$-Laplace equation, total variation flow, fast diffusion equation, ergodic semigroup, unique invariant measure, variational convergence, $e$-property}
 \subjclass[2000]{60H15; 35K67, 37L40, 49J45}

\maketitle

We have decided to write this corrigendum because of the emergence of a new result, credited to
H. Br\'ezis, and worked out by V. Barbu and M. R\"ockner \cite[Section 8, Appendix 1]{BR12}. In our work \cite{CiotToe}, we claim and make use of the validity of \cite[Eq. (4.19)]{BDPR} (e.g. as the necessary ingredient for \cite[``Step 1'' in the proof of Theorem 4.4, p. 2011]{CiotToe}). The basic arguments to verify \cite[Eq. (4.19)]{BDPR} (which is \eqref{eq:brezis} below) turned out to be wrong, even for smooth domains. Let us present the result of H. Br\'ezis and discuss the consequences
for our work. For a detailed proof, we refer the reader to \cite[Section 8, Appendix 1]{BR12}.

\begin{prop}\label{prop:brezis}
Let $\Lambda$ be a bounded, convex domain of $\R^d$, $d\ge 1$, with piecewise smooth
boundary $\partial\Lambda$ of class $C^2$. Let $J_n=(1-\frac{\Delta}{n})^{-1}$, $n\in\N$, be the resolvent
of the Dirichlet Laplacian $(-\Delta,D(-\Delta))$, where $D(-\Delta)=H^{1}_0(\Lambda)\cap H^2(\Lambda)$. Then
\begin{equation}\label{eq:brezis}\int_\Lambda\abs{\nabla J_n(u)}\,d\xi\le\int_\Lambda\abs{\nabla u}\,d\xi,\quad\forall u\in W^{1,1}_0(\Lambda),\;n\in\N.\end{equation}
\end{prop}
\begin{proof}
See \cite[Proposition 8.1, Remark 8.4]{BR12}.
\end{proof}

The main consequence of Proposition \ref{prop:brezis}, or, to be more precise, of the lack of a proof for a more general situation, is that Theorem 4.4 in \cite{CiotToe} merely holds for
bounded, convex domains $\Lambda$ with piecewise smooth
boundary $\partial\Lambda$ of class $C^2$ and not (yet), as claimed by us,
for a general Lipschitz boundary.

At another point of our work, in the proof of \cite[Theorem 3.2]{CiotToe}, we use a Krylov-Bogoliubov-type
argument for the convergence of invariant measures. The passage to the limit,
however, remains nebulous. Let us remark that the statement of \cite[Theorem 3.2]{CiotToe}
follows, combined with the (uniform) tightness of the invariant measures proved in \cite[proof of Theorem 3.2]{CiotToe}, from the more general Proposition \ref{prop:conv} below.

Let $(\Xscr,d)$ be a Polish space with complete metric $d$.

\begin{prop}\label{prop:conv}
Let $\{P_t^n\}_{t\ge 0}$, $n\in\N$, be Feller semigroups on $(\Xscr,d)$ with invariant measures $\mu_n$, $n\in\N$, respectively.
Suppose that $\{\mu_n\}_{n\in\N}$ has a weakly converging subsequence and $\{P_t^n\}$ converges to a Feller semigroup $\{P_t\}_{t\ge 0}$ in the following sense:
For all $t\ge 0$, $\psi\in\Lip_b(\Xscr)$, $x\in\Xscr$ we have that
\[\lim_n P_t^n\psi(x)=P_t\psi(x).\]

Suppose that the $\{P_t^n\}$ have a uniform ``\emph{Lipschitz-type'' $e$-property},
that is,
\begin{multline*}\exists C>0:\;\forall\psi\in\Lip_b(\Xscr),\;\forall x\in\Xscr\;\forall z\in\Xscr,\\
\forall t\ge 0,\;\forall n\in\N,\quad\abs{P_t^n\psi(x)-P_t^n\psi(z)}\le C\Lip(\psi)\, d(x,z).\end{multline*}
If $\{P_t\}$ admits a unique
invariant measure $\mu$, then $\mu_n\to\mu$ weakly.
\end{prop}
\begin{proof}
Let $\{\mu_{n_k}\}$ be a subsequence of $\{\mu_n\}$ such
that $\mu_{n_k}\to\mu_0$ weakly as $k\to\infty$ where $\mu_0$ is some probability measure. For simplicity, let us just
write $\{\mu_n\}$. Recall that weak convergence
of probability measures
is metrizable with the so-called \emph{bounded Lipschitz metric} defined by
\[\beta(\nu_1,\nu_2):=\sup\left\{\lrabs{\int_\Xscr\psi\, d(\nu_1-\nu_2)}\;\bigg\vert\;\psi\in\Lip_b(\Xscr),\;\norm{\psi}_\infty+\Lip(\psi)\le 1\right\},\]
compare with \cite[1.12, pp. 73/74]{vdVW96}.

If we can prove that $\mu_0$ is an invariant measure of $\{P_t\}_{t\ge 0}$,
we are done and the whole sequence $\{\mu_n\}$ converges to $\mu=\mu_0$. Let
$\psi\in\Lip_b(\Xscr)$, $t\ge 0$, and
\begin{align*}
&\lrabs{\int_\Xscr \psi\,d\mu_0-\int_\Xscr P_t\psi\,d\mu_0}\\
\le&\lrabs{\int_\Xscr \psi\,d\mu_0-\int_\Xscr P_t^n\psi\,d\mu_n}+\lrabs{\int_\Xscr P_t^n \psi\,d\mu_n-\int_\Xscr P_t^n\psi\,d\mu_0}
\\&+\lrabs{\int_\Xscr P_t^n \psi\,d\mu_0-\int_\Xscr P_t\psi\,d\mu_0}
\end{align*}
By invariance, the first term equals
\[\lrabs{\int_\Xscr \psi\,d\mu_0-\int_\Xscr\psi\,d\mu_n}\]
and hence converges to zero as $n\to\infty$.

The third term converges to zero by the convergence of semigroups and
Lebesgue's dominated convergence, since the integrand is bounded
by $2\norm{\psi}_\infty$.

Let us investigate the second term:
\begin{align*}&\lrabs{\int_\Xscr P_t^n \psi\,d\mu_n-\int_\Xscr P_t^n\psi\,d\mu_0}\\
 \le&\beta(\mu_n,\mu_0)\left[\norm{P_t^n\psi}_\infty+\Lip(P_t^n\psi)\right]\\
\le&\beta(\mu_n,\mu_0)\left[\norm{\psi}_\infty+C\Lip(\psi)\right]\\
\longrightarrow\,& 0\quad\text{as}\;n\to\infty.
\end{align*}
Hence $\lrabs{\int_\Xscr \psi\,d\mu_0-\int_\Xscr P_t\psi\,d\mu_0}=0$ for
all $\psi\in\Lip_b(\Xscr)$ and all $t\ge 0$ and so $\mu_0$ is invariant
for $\{P_t\}$.
\end{proof}

Of course in our work, $(\Xscr,d)$ is the Hilbert space
$L^2(\Lambda)$ or $H^{-1}(\Lambda)$.
We note that in \cite{CiotToe}, the Lipschitz-type $e$-property (with $C=1$) can
be verified for the approximating semigroups by It\={o}'s formula
and the monotonicity of the operators involved (not so for the case
$p=1$, which is not needed, though).

Furthermore, we have conjectured the convergence of invariant measures for
the case $p=1$, namely in \cite[Conjecture 4.5]{CiotToe}. Now, the conjecture has been solved by B. Gess and the second-named author; see \cite[Section 7]{GesToe}.

The use of the wrong energy in \cite{BDPR}, as indicated by us in
\cite[Remark 4.3]{CiotToe}, was corrected by V. Barbu, G. Da Prato and M. R\"ockner in
\cite{BDPR12}. In fact, the lower semi-continuous envelope of
the 1-Laplacian energy is well-studied and characterized in e.g.
\cite[Proposition 11.3.2, p. 438]{ABM}.

\subsection*{Some typographical errors}

\begin{itemize}
\item
In \cite{CiotToe},
on p. 2000, the sixth line from the bottom, $(W^{1,p})^\ast$ should clearly
appear as a subscript.
\item In \cite[Definition 4.2, p. 2008]{CiotToe}, it should read
$u$ instead of $x$ in the third line.
\item In \cite[Remark 4.3, p. 2009]{CiotToe}, in line 7 it should read
\[\varliminf_n\Psi(u_n)=\operatorname{Per}(\Lambda,\R^d)<+\infty=\Psi(\mathbbm{1}_\Lambda),\]
where $\operatorname{Per}(\Lambda,\R^d)$ is the \emph{perimeter}.
In particular, for dimension $d=1$, this implies
\[\varliminf_n\Psi(u_n)=2<+\infty=\Psi(\mathbbm{1}_\Lambda).\]
\item At the end of \cite[proof of Lemma 4.7, p. 2010]{CiotToe},
it should read $\Psi_\eps^p${'}{s} instead of $\Psi_\eps^p${\'s}.
\end{itemize}

\end{document}